\documentclass[12pt]{article}
\usepackage{amsmath,amsfonts,amssymb,amsthm}
\usepackage[pdftex]{graphicx} 
\usepackage[all]{xy}
\usepackage{multirow}
\usepackage{framed}
\usepackage[dvipsnames]{color}
\usepackage[T1]{fontenc}
\usepackage{times} 
\usepackage[cp1250]{inputenc}  
\def\Z{{\mathbb Z}}
\def\C{{\mathbb C}}
\def\R{{\mathbb R}}
\def\acts{\triangleright}
\def\T{{\mathbb T}}
\def\Q{{\mathbb Q}}

\def\CA{{\mathcal A}}

\def\qq{{\quad}}
\DeclareMathSymbol\crossrt{\mathrel}{AMSb}{"6E}
\DeclareMathSymbol\crosslt{\mathrel}{AMSb}{"6F}
\def\lcross{\crossrt}
\def\lcross{\crosslt}
\def\id{{\mathrm i}{\mathrm d}}
\def\ts{\otimes}
\def\oh{\frac{1}{2}}
\def\soh{\hbox{\small $\frac{1}{2}$}}

\def\im{\hbox{Im\ }}

\def\ker{\hbox{ker\ }}
\def\cop{\Delta}

\def\acts{\triangleright}

\def\ts{\otimes}

\def\im{\hbox{Im\ }}
\def\ker{\hbox{ker\ }}

\newtheorem{lemma}{Lemma}[section]
\newtheorem{theorem}[lemma]{Theorem}
\newtheorem{proposition}[lemma]{Proposition}
\newtheorem{corollary}[lemma]{Corollary}
\newtheorem{definition}[lemma]{Definition}

\newtheorem{remark}[lemma]{Remark}

\newcommand{\be}{\begin{equation}}
\newcommand{\ee}{\end{equation}}
\newcommand{\bs}{\begin{frame}}
\newcommand{\esl}{\end{frame}}
\newcommand{\bea}{\begin{eqnarray}}
\newcommand{\eea}{\end{eqnarray}}

\definecolor{lgray}{rgb}{0.2,0.2,0.2}
\definecolor{Light}{rgb}{0.8,0.8,0}
\definecolor{Brown}{cmyk}{0, 0.8, 1, 0.6}
\definecolor{Yellow}{rgb}{1, 1, 0}
\definecolor{Dark}{gray}{.20}
 
\definecolor{lgray}{rgb}{0.2,0.2,0.2}
\definecolor{shadecolor}{named}{GreenYellow} 
\title{$K$-theory of noncommutative \\ Bieberbach manifolds.}
\author{
\vbox{
\small
\begin{center} 
\large 
Piotr Olczykowski \thanks{supported by the grant from The John Templeton Foundation} \\
{\em and}  \\
Andrzej Sitarz 
\thanks{Partially supported by MNII grant 189/6.PRUE/2007/7  and NCN grant 2011/01/B/ST1/06474} \\
\ \\
Institute of Physics, Jagiellonian University,\\
Reymonta 4, 30-059 Krak\'ow, Poland \\
\end{center} 
}}
\begin{document}
\maketitle
\begin{abstract}
We compute $K$-theory of noncommutative Bieberbach manifolds, which quotients of 
a three-dimensional noncommutative torus by a free action of a cyclic group $\Z_N$, 
$N=2,3,4,6$.
\end{abstract}

\section{Introduction}

Bieberbach manifolds are compact manifolds, which are quotients of the Euclidean space by 
a free, properly discontinuous and isometric action of a discrete group. The classification 
of all Bieberbach manifolds is a complex problem in itself (see \cite{Sad} and the references 
therein). The first nontrivial low-dimensional examples appear in dimension three and
have been already described in the seminal works of Biebierbach \cite{Bieb1,Bieb2}.

In the paper we work with the $C^\ast$ algebras of continuous functions on the three-torus 
and the corresponding noncommutative deformation $C(\T^3_\theta)$. 

\subsection{Three-dimensional Bieberbach Manifolds}

In this section we shall briefly recall the description of three-dimensional Bieberbach manifolds 
as quotients of the three-dimensional tori by the action of a finite discrete group. We use the 
algebraic language, taking the algebra of the polynomial functions on the three-torus $\T^3$ 
generated by three mutually commuting unitaries $U,V,W$. This algebra could be then completed 
first to the algebra of smooth functions on the torus $C^\infty(\T^3)$ and later to a 
$C^\ast$-algebra of continuous functions $C(\T^3)$.

There are, in total, 10 different Bieberbach three-dimensional manifolds, six orientable (including the 
three-torus itself) and four nonorientable ones. This follows directly from the classification of Bieberbach groups of $\R^3$ out of which six do not change orientation and four change the orientation. The action of the finite groups on the unitaries $U,V,W$, which generate the algebra of continuous functions on the three-torus is presented in the table below and comes directly from the action of Bieberbach groups on $\R^3$.  
%%%%%%%%%%%%%%%%%%%%%%%%%%%%%%%%%%%%%%%%%%%%%%%%%%%%
\begin{table}[here]
\centering
\begin{tabular}{|c|c|c|c|}
\hline
name & group $G$ & generators of $G$ & action of $G$ on $U,V,W$  \\ \hline \hline
B2& $\Z_2$ & e & $e \acts U = -U$, $e \acts V = V^*$, $e \acts W = W^*$ \\ \hline
B3& $\Z_3$ & e & $e \acts U =  e^{\frac{2}{3}\pi i} U$, $e \acts V = W^*$, $e \acts W = W^* V$ \\ \hline
B4& $\Z_4$ & e & $e \acts U =  i U$, $e \acts V = W$, $e \acts W = V^*$ \\ \hline
B5& $\Z_2 \times \Z_2$ & $e_1,e_2$ & $e_1 \acts U =  -U$, $e_1 \acts V = V^*$, $e_1 \acts W = W^*$ \\
& & & $e_2 \acts U =  U^*$, $e_2 \acts V = -V$, $e_2 \acts W = -W^*$ \\ \hline
B6& $\Z_6$ & e & $e \acts U =  e^{\frac{1}{3}\pi i} U$, $e \acts V = W$ , $e \acts W = W V^*$ \\ \hline
\end{tabular}
\caption{Orientable actions of finite groups on three-torus}
\label{act_o}
\end{table}
%%%%%%%%%%%%%%%%%%%%%%%%%%%%%%%%%%%%%%%%%%%%%%%%%%%%%%%%%%%%%%%%%%%%%%%%%%%%%%%%%%%%%%%%%%%%%%
The above actions give rise to five oriented flat three-manifolds different from the torus. 
The remaining four nonorientable quotients, originate from the following actions:
%%%%%%%%%%%%%%%%%%%%%%%%%%%%%%%%%%%%%%%%%%%%%%%%%%%%%%%%%%%%%%%%%%%%%%%%%%%%%%%%%%%%%%%%%%%%%%
\begin{table}[here]
\centering
\begin{tabular}{|c|c|c|c|}
\hline
name & group $G$ & generators of $G$ & action of $G$ on $U,V,W$, \\ \hline \hline
N1& $\Z_2$ & $e$ & $e \acts U = -U$, $e \acts V = V$, $e \acts W = W^*$ \\ \hline
N2& $\Z_2$ & $e$ & $e \acts U = -U$, $e \acts V = V W$, $e_1 \acts W = W^*$ \\ \hline
N3& $\Z_2 \times \Z_2$ & $e_1,e_2$ & $e_1 \acts U = -U$, $e_1 \acts V = V^*$,   $e_1 \acts W = W^*$ \\ 
  &        &         &               $e_2 \acts U =  U$, $e_2 \acts V =-V    $, $e_2 \acts W = W^*  $ \\ \hline
N4& $\Z_2 \times \Z_2$ & $e_1,e_2$ & $e_1 \acts U = -U$, $e_1 \acts V = V^*$,   $e_1 \acts W = W^*$ \\ 
  &        &         &               $e_2 \acts U =  U$, $e_2 \acts V =-V    $, $e_2 \acts W = -W^*  $ \\ \hline
\end{tabular}
\caption{Nonorientable action of finite groups on three-torus}
\label{act_no}
\end{table}

For full details and classifications of all free actions of finite groups on three-torus see \cite{HJKL93,HJKL02}, note, however, that the resulting quotient manifolds are always 
one of the above Bieberbach manifolds.
%%%%%%%%%%%%%%%%%%%%%%%%%%%%%%%%%%%%%%%%%%%%%%%%%%%%%%%%%%%%%%%%%%%%%%%%%%%%%%%%%%%%%%%%%%%%%%
It is easy to see that ${\mathrm N1}$ is just the Cartesian product of $S^1$ with the Klein bottle, whereas ${\mathrm N3}$ and ${\mathrm N4}$ are two distinct $\Z_2$ quotients of ${\mathrm B2}$. 
%%%%%%%%%%%%%%%%%%%%%%%%%%%%%%%%%%%%%%%%%%%%%%%%%%%%%%%%%%%%%%%%%%%%%%%%%%%%%%%%%%%%%%%%%%%%%%
\section{Noncommutative Bieberbach manifolds}

Since a convenient description of Bieberbach manifolds is as quotients of three-torus by action 
of finite groups we can ask whether starting from a noncommutative three-dimensional torus we 
can obtain interesting examples of nontrivial flat noncommutative manifolds. The most general
noncommutative three-torus can be realized as a twisted group algebra 
$C^*(\Z^3,\omega_\theta)$ with the cocycle over $\Z^3$:
$$ \omega_\theta(\vec{m}, \vec{n}) = e^{\pi i \sum_{j,k =1}^3 \theta_{jk} m_j n_k}, \quad
\vec{m},\vec{n} \in \Z^3, $$
where $\theta_{jk}$ is a real antisymmetric matrix ($0{\leq}\theta_{jk}<1$). 
Taking the canonical basis of $\Z^3$, $\vec{e}_1,\vec{e}_2,\vec{e}_3$, we 
can identify $U,V,W$ with $\delta_{\vec{e}_1}, \delta_{\vec{e}_2},\delta_{\vec{e}_2}$ 
in the convolution
 algebra $C^*(\Z^3)$. We denote the product in the twisted convolution
algebra by $*_\omega$, and on the generators we have
$$ \delta_{\vec{e}_i} *_{\omega_\theta} \delta_{\vec{e}_j} = 
   \omega_\theta(\vec{e}_i,\vec{e}_j) \delta_{\vec{e}_i} \delta_{\vec{e}_j}. $$

Next, we shall find all possible values of the matrix $\theta_{jk}$ such that the actions of 
the finite group $G$ (as discussed earlier) are compatible with the cocycle. As the action 
of $G$ is not an action on $\Z^3$ then the cocycle cannot be in fact invariant. However, we 
might define the compatibility with the action of $G$ in the following way. We say that the
action of the finite group $G$ is compatible with the cocycle $\omega_\theta$ if:
$$ g \acts (a *_{\omega_\theta} b) = (g \acts a) *_{\omega_\theta} (g \acts b), 
\quad \forall a,b \in C^*(\Z^3), g \in G.$$ 
We have:
\begin{lemma}
The cocycle $\omega_\theta$ is compatible with the actions of group $G$, given in the tables 
\ref{act_o} and \ref{act_no} if $0 \leq \theta_{jk} < 1$ are as follows:

\begin{table}[here]
\centering
\begin{tabular}{|c|c|c|c|c|}
\hline
group $G$ & $\theta_{12}$ & $\theta_{13}$
 & $\theta_{23}$ & conditions \\ \hline \hline
$\Z_2$ & $\frac{k}{2}$ & $\frac{l}{2}$   & $\theta$ & $k,l=0,1,$   \\ \hline
$\Z_3$ & $\frac{k}{3}$ & $\frac{3-k}{3}$ & $\theta$ & $k=0,1,2,$ \\ \hline
$\Z_4$ & $\frac{k}{2}$ & $\frac{k}{2}$   & $\theta$ & $k=0,1,$   \\ \hline
$\Z_2 \times \Z_2$ 
       & $\frac{k}{2}$ & $\frac{l}{2}$ & $\frac{m}{2}$ & $k,l,m=0,1$ \\ \hline
$\Z_6$ & $\frac{k}{3}$ & $\frac{3-k}{3}$ & $\theta$ & $k=0,1,2,$ \\ \hline 
\end{tabular}
\caption{Values of $\theta_{jk}$ for compatible cocycles for orientable actions}
\label{coc_o}
\end{table}
%%%%%%%%%%%%%%%%%%%%%%%%%%%%%%%%%%%%%%%%%%%%%%%%%%%%%%%%%%%%%%%%%%%%%%%%%%%%%%%%%%%%%%%%%%%%%%%%%%
\begin{table}[here]
\centering
\begin{tabular}{|c|c|c|c|c|}
\hline
group $G$ & $\theta_{12}$ & $\theta_{13}$ & $\theta_{23}$ & conditions \\ \hline \hline
$\Z_2$ & $\theta$ & $\frac{k}{2}$ & $\frac{l}{2}$  & $k,l=0,1$,   \\ \hline
$\Z_2$ 
       & $\theta$ & $\frac{k}{2}$ & $\frac{l}{2}$ &$k,l=0,1$ \\ \hline
$\Z_2 \times \Z_2$ 
       & $\frac{k}{2}$ & $\frac{l}{2}$ & $\frac{m}{2}$ & $k,l,m=0,1$ \\ \hline
$\Z_2 \times \Z_2$ 
       & $\frac{k}{2}$ & $\frac{l}{2}$ & $\frac{m}{2}$ & $k,l,m=0,1$ \\ \hline
\end{tabular}
\caption{Values of $\theta_{jk}$ for compatible cocycles for nonorientable actions}
\label{coc_no}
\end{table}
\end{lemma}
As we are interested in the genuine noncommutative case, where the three-dimensional 
torus has at least one irrational rotation subalgebra (that is, at least one of the 
independent entries of the matrix $\theta_{jk}$ is irrational), we see that we might 
obtain only $4$ nontrivial orientable noncommutative Bieberbach manifolds and two 
nonorientable ones. 

\begin{definition}
Let $C(\T^3_\theta)$, be a twisted group algebra over $\Z^3$ corresponding to a cocycle 
obtained from $\theta_{12}=\theta_{21}=0$ and $\theta_{23}=-\theta$ for an irrational 
$0 < \theta < 1$. Then the generating unitaries $U,V,W$ satisfy relations:
$$ UV=VU, \quad UW=WU, \quad WV=e^{2\pi i\theta}VW. $$

We define the algebras of noncommutative Bieberbach manifolds as the fixed point 
algebras of the following actions of finite groups $G$ on $C(\T^3_\theta)$ (note that for
$\mathrm{N1}_\theta$ and $\mathrm{N1}_\theta$ we need to relabel the generators:  $\{U,V,W\} \to \{W, U, V \}$ so that always the $V$ and $W$ are from the irrational rotation subalgebra), which
are combine in the table \ref{table1}. For convenience and to match the notation of other 
papers we rescaled the generators $V,W$ in the case of $\Z_3$ and $\Z_6$ actions, 
also, for $\Z_3$ we take the other generator of the $\Z_3$ action. 
%%%%%%%%%%%%%%%%%%%%%%%%%%%%%%%%%%%%%%%%%%%%
%%%%%%%%%%%%%%%%%%%%%%%%%%%%%%%%%%%%%%%%%%%%
\begin{table}[here]
\centering
\begin{tabular}{|c|c|c|c|l|}
\hline
name & group $\Z_N$ & $e^N=\id$ & action of $\Z_N$ on $U,V,W$  \\ \hline \hline
$\mathrm{B2}_\theta$& $\Z_2$ & e & $e \acts U = -U$, $e \acts V = V^*$, $e \acts W = W^*$, \\ \hline
$\mathrm{B3}_\theta$& $\Z_3$ & e & $e \acts U =  e^{\frac{2}{3} \pi i} U$, $e \acts V = e^{-\pi i \theta} V^* W$, 
                   $e \acts W = V^*$, \\ \hline
$\mathrm{B4}_\theta$& $\Z_4$ & e & $e \acts U =  i U$, $e \acts V = W$, $e \acts W = V^*$, \\ \hline
$\mathrm{B6}_\theta$& $\Z_6$ & e & $e \acts U =  e^{\frac{1}{3}\pi i} U$, $e \acts V = W$ , $e \acts W = e^{-\pi i \theta} V^*W$, \\ \hline
$\mathrm{N1}_\theta$& $\Z_2$ & e & $e \acts U = U^*$, $e \acts V = -V$, $e \acts W = W$, \\ \hline
$\mathrm{N2}_\theta$& $\Z_2$ & e & $e \acts U = U^*$, $e \acts V = -V $, $e \acts W = W U^*$, \\  \hline
\end{tabular}
\caption{The action of finite cyclic groups on $C(T^3_\theta)$}
\label{table1}
\end{table}
\end{definition}

\begin{proposition}
The actions of the cyclic groups $\Z_N$, $N=2,3,4,6$ on the noncommutative 
three-torus, as given in the table \ref{table1} is free.
\end{proposition}
\begin{proof}
Let us recall that that the freeness of a coaction of a Hopf algebra $H$ on a $C^*$-algebra  $A$ means (for a right coaction) that the spans of $ (a \ts \id) \cop (b) $  and $ \cop(b) (a \ts \id)$, $a,b \in A$ are dense in $A \ts H$ for a minimal tensor product. 

In our case, the freeness of the action is easy to verify. In the B2,B3,B4,B6 case the coaction of the dual Hopf algebra to the group algebra $\Z_N$ is simply: $\cop U = U \ts \tilde{e}$, where $\tilde{e}$ is the generator
of $\C(\Z_N)$. Since $U$ (and its powers) are invertible, it is evident that $(a \ts \id) \cop(U^n)$ and
 $\cop(U^n) (a \ts \id)$ are dense in $A \ts \C(\Z_N)$. In the ${\mathrm N1}_\theta, {\mathrm N2}_\theta$ case the same argument applies when we take $V$ instead.
\end{proof}

%%%%%%%%%%%%%%%%%%%%%%%%%%%%%%%%%%%%%%%%%%%%%%%%%%%%%%%%%%%%%%%%%%%%%%%%%%%%%%%%%%%%%%
The purpose of this paper is to compute K-theory of noncommutative Bieberbach algebras 
$\mathrm{B2}_\theta,\mathrm{B3}_\theta,\mathrm{B4}_\theta,\mathrm{B6}_\theta$. 

\section{K-theory for Noncommutative Bieberbach manifolds} 
 
We present here the computation of the K-theory groups from the Bieberbach manifolds obtained first by 
the action of the cyclic group $\Z_N$, $N=2,3,4,6$, leaving the $\Z_2 \times \Z_2$ case aside for future considerations. 

As in our case the action of cyclic group fulfills the assumptions of Takai duality \cite{KiTa,Taka} the fixed point algebra is Morita equivalent to the crossed product algebra $C(\T^3_\theta) \rtimes \Z_N$,  $N=2,3,4,6$. In fact, it is easy to show an explicit isomorphism between
$C(T^3_\theta)^{\Z_n} \otimes M_N(\C)$ and $C(\T^3_\theta) \lcross \Z_N$, which we 
present explicitly.

\begin{lemma}
Consider the action of a finite group $\Z_N$ on $C(\T^3_\theta)$ determined by the action
of the generator $p$:
$$ p \acts U = \lambda U, \qq\qq p \acts V = \alpha(V), \qq\qq p \acts W = \alpha(W), $$
where $U$ is central, $\lambda = e^{\frac{2\pi i}{N}}$ and $\alpha$ is an automorphism of the rotation algebra over generators $V$ and $W$ (with $\theta$ not necessarily irrational).  
The crossed product algebra  $C(\T^3_\theta) \lcross \Z_N$ is generated by $U,V,W$ and $p$ with relations:
$$ pU = \lambda Up, \qq pV =\alpha(V)p, \qq pW=\alpha(W)p, \qq p^N=1. $$
is then canonically isomorphic to  $C(T^3_\theta)^{\Z_n} \otimes M_N(\C)$.
\end{lemma}
    
\begin{proof}
Consider an element $\hat{p}$ in the crossed product algebra:

\be \hat{p} = U + \left( \frac{1}{N}  \sum_{k=1}^N p^k \right) \left( U^{1-N} - U \right). 
\label{up}
\ee
It is easy to verify that:
$$ \hat{p}^N =1, \qq\qq p \hat{p} = \lambda \hat{p} p, $$
so $p$ and $\hat{p}$ generate the matrix algebra $M_N(\C)$. Since $U$ is central
both $p$ and $\hat{p}$ commute with any element of the fixed point subalgebra 
$C(\T^3_\theta)^{\Z_n}$.

We shall demonstrate now the isomorphism from the lemma, which we shall denote
by $\Psi$. First,  the relation (\ref{up}) could be inverted, yielding:
$$ \Psi(U) = \hat{p} + \left( \frac{1}{N}  \sum_{k=1}^N p^k \right) \hat{p} 
\left(  U^N - 1 \right). $$
Since $U^N$ is an invariant element of the algebra then $\Psi(U)$ is clearly in 
$C(\T^3_\theta)^{\Z_N} \otimes M_N(\C)$. Take now arbitrary 
$x \in C(T^3_\theta)$. It is easy to see that $x$ could be uniquely decomposed
as a sum of elements homogeneous with respect to the action of $\Z_N$:
$$ x = \sum_{k=0}^{N-1} x_k, \qq\qq p \acts x_k = \lambda^k x_k. $$
Indeed, it is sufficient to take:
$$ x_k = \frac{1}{N} \sum_{j=0}^{N-1} \bar{\lambda}^{kj} (p^j \acts x), \;\; k=0,\ldots,N-1. $$
Then if we define:
$$ \Psi(x) = \sum_{k=0}^{N-1} (x_k U^{-k} ) \Psi(U^k), $$
then the range of $\Psi$ is clearly in $C(\T^3_\theta)^{\Z_N} \otimes M_N(\C)$, since
each of the elements $x_k U^{-k}$ is invariant and in the fixed point algebra. The
verification that $\Psi$ is an algebra morphism and is an isomorphism is left to 
the reader.
\end{proof}

Note that the isomorphism $\Psi$ which provides the Morita equivalence in our case
does not depend on the value of the parameter $\theta$, hence for the Bieberbach
manifolds (commutative and noncommutative) their $K$-Theory groups are the 
same as the $K$-theory groups of the crossed product algebras. 

A technical tool for the computations is the following lemma.
\begin{lemma}\label{across}
Let $\CA$ be a $C^*$-algebra, $\beta$ its automorphism and let $\alpha$
denote an action of $\Z_N$ on $\CA \rtimes_\beta \Z$, such that the 
restriction of the action on $\C\Z$ is by multiplication by a root of unity,
$\lambda^N = 1$ on its generator. 
Then the algebra $(\CA \rtimes_\beta \Z )\rtimes_\alpha\Z_N$ is 
isomorphic to $(\CA \rtimes_\alpha \Z_N) \rtimes_{\hat{\beta}} \Z$,
where $\hat{\beta}$ is the action of $\Z$, which is $\beta$ on $\CA$
and multiplication by a $\bar{\lambda}$ on $\Z_N$.
\end{lemma}
\begin{proof}
With the notation above, let us denote
 by $U$ the generator of $\Z$ and 
by $e$ the generator of $\Z_N$.  We have:
$$ U \acts a = \beta(a), \;\;\; e \acts a = \alpha(a), \;\;\; 
    e \acts U =  \lambda U. $$
The action $\hat{\beta}$ is defined as follows:
\begin{equation}
\hat{\beta}(a) = \beta(a), \;\;\; \forall a \in \CA, 
\qq\qq \hat{\beta}(e) = \overline{\lambda} e.
\label{bhat}
\end{equation}
It is easy to see that both crossed product algebras are isomorphic
to each other as the defining relations are identical. 
\end{proof}

Applying this to the case of the $T^3_\theta$ and cross product by the action
of $\Z_N$ (N=2,3,4,6) we have:

\begin{corollary}
The algebra of the noncommutative three-torus $C(T^3_\theta)$ is a crossed product 
$C(T^2_\theta) \lcross \Z$. For $N=2,3,4,6$ the action $\alpha$ of $\Z_N$ on it is 
by multiplication on the generator of $\Z$ and leaves the algebra $C(T^2_\theta)$ 
invariant. This action comes, in fact, from the $SL(2,\Z)$ group of automorphisms of 
$T^2_\theta$. Therefore by the lemma \ref{across} we have for $N=2,3,4,6$ the
following isomorphism:
\begin{equation}
C(T^3_\theta) \rtimes_\alpha \Z_N \sim (C(T^2_\theta) \rtimes_\alpha \Z_N) 
\rtimes_{\hat{\beta}} \Z, 
\label{actionZ}
\end{equation}
where $\hat{\beta}(a) = a, \forall a \in C(\T^2_\theta)$ and $\hat{\beta}(e) = \bar{\lambda} e$,
for $e \in \C\Z_n$ (generator of $\Z_N$).
\end{corollary}

\begin{remark}
The crossed product algebras of the noncommutative torus by the cyclic subgroups 
of $SL(2,\Z)$ have been studied intensely as symmetric noncommutative tori and 
noncommutative spheres \cite{ELPW}. Although classically they correspond to 
orbifolds rather than to manifolds, we can nevertheless view the noncommutative 
Bieberbach algebras as circle bundles over some noncommutative spheres.
\end{remark}

\subsection{$K$-theory groups from Pimsner-Voiculescu}

Using (\ref{actionZ}) we can use the Pimsner-Voiculescu exact sequence provided that we 
know the $K$-theory groups of the corresponding crossed product of noncommutative torus 
by the actions of the respective cyclic group $\Z_N$ and the exact form of the action 
of $\Z$ on the generators of these $K$-theory groups. Although we shall proceed case 
by case the methods are basically identical: we first determine the generators 
of $K$-theory groups for the algebras $C(T^2_\theta) \lcross \Z_N$, $N=2,3,4,6$and 
the explicit action of the $\hat{\beta}$ automorphisms on them. 

The basic tool is the existence of traces on the dense subalgebra of the crossed product
algebra of noncommutative torus by a finite cyclic group and their behavior under the 
action of $\hat{\beta}$. The origin of such traces is easy to understand, let us recall
that in fact they come from twisted traces on the algebra of the noncommutative torus 
itself. 
\begin{remark}
Let $\CA$ be an algebra and let $\sigma$ denote an action of a finite cyclic group $\Z_N$. 
If $\Phi_s$ is an $\sigma$-invariant and $\sigma^s$-twisted trace on $\CA$, 
$0 < s \leq N$: 
$$ \Phi_s(\sigma(a))= \Phi_s(a), \qq\qq\qq \Phi_s(ab) = \Phi_s(\sigma^s(b)a), \qq 
\forall a,b \in \CA,$$
then $\Phi$ extends to a trace on the crossed product algebra $\CA \lcross \Z_N$:
$$ \tilde{\Phi}_s \left( \sum_{k=0}^{N-1} a_k e^k \right) = \Phi_{s} (a_{N-s}), 
     \qq 0 < s\leq N.$$
\end{remark}
where $e$ is the generator of $\Z_N$. The proof of the fact is a simple computation:
$$
\begin{aligned}
\tilde{\Phi}_s & \left( \left( \sum_{k=0}^N a_k e^k \right) \left( \sum_{j=0}^N b_j e^j \right) \right)
  = \tilde{\Phi}_s \left( \sum_{k,j=0}^N a_k \sigma^{k} (b_j) e^{k+j} \right) \\
& = \sum_{k+j=N-s} \Phi_s \left( a_k \sigma^{k}(b_j) \right)  
  = \sum_{k+j=N-s} \Phi_s \left( \sigma^{k+s}(b_j) a_k \right) 
  = \sum_{k+j=N-s} \Phi_s \left( b_j \sigma^{j}(a_k) \right) \\
& = \tilde{\Phi}_s  \left( \left( \sum_{k=0}^N b_j e^j \right) \left( \sum_{j=0}^N a_k e^k \right) \right).
\end{aligned}
$$

In a series of papers Walters and Buck and Walters demonstrated the following crucial theorem:
\begin{theorem}
\label{injectK0}
Let $\T^2_\theta$ be the irrational rotation algebra. Then for $G=\Z_N$,$N=2,3,4,6$ with the 
actions (on the generators $V,W$) given in table (\ref{table1}) there exists a family of unbounded 
traces on the algebra $\T^2_\theta \lcross G$, which together with the canonical trace 
$\tau$ on $\T^2_\theta$ provide an injective morphism from the $K_0$-group 
into $\C^{r(N)}$, for some $r(N)$.
\end{theorem}

The proofs and the exact form of these traces and their value on the generators of $K_0$-group
are to be found in \cite{Walt95,Walt00, BuWa1, BuWa2}. We skip the presentation of 
details, showing as an illustration an example of the easiest $N=2$ case. Following 
\cite[page 592]{Walt95} we see that there are four unbounded traces on 
$T^2_\theta \rtimes \Z_2$: $\tau_{jk}$, $j,k= 0,1$, which are defined as follows 
on the basis of $\T^2_\theta \lcross \Z_2$:
\begin{equation}
\tau_{jk} ( V^\iota W^\kappa p^\rho) = 4 e^{-\pi i \theta \iota \kappa} \delta^\rho_1 
\delta^{\bar{\iota}}_j \delta^{\bar{\kappa}}_k, \qq\qq \iota,\kappa \in \Z, \rho=0,1,
\end{equation}
where $\bar{x} = x \mod 2$. The other cases ($N=3,4,6$) can be treated similarly. 
Note that the collection of traces provides no longer an injective map from $K_0$
into $\C^{r(N)}$. However, if one adds the nontrivial cyclic two-cocycle then it is
again in injective morphism. 

To have all necessary tools we only need to study the behavior of the 
traces under the action of $\hat{\beta}$.
\begin{lemma}
\label{betatrace}
If $\tilde{\Phi}_s$ is a trace on $C(\T^2_\theta) \lcross \Z_N$, which comes from 
a $\sigma^s$-twisted trace, then under the action of $\Z$ by $\hat{\beta}$ we have: 
\begin{equation}
\tilde{\Phi}_s(\hat{\beta}(a)) = e^{\frac{2\pi i s}{N}} \tilde{\Phi}_s(a),
\qq 
\forall a \in \T^2_\theta \lcross \Z_N 
\end{equation}
\end{lemma}

\begin{proof}
The above property follows directly from the form of the action $\hat{\beta}$ (\ref{bhat}).
$$ 
\begin{aligned}
\tilde{\Phi}_s & \left( \hat{\beta} \left( \sum_{k=0}^N a_k e^k \right) \right) 
 =  \tilde{\Phi}_s \left( \sum_{k=0}^N a_k \bar{\lambda}^k e^k \right) \\
 & = \bar{\lambda}^{N-s} \Phi_s (a_{N-s}) =  e^{\frac{2 \pi i s}{N}} \tilde{\Phi}_s  
\left( \sum_{k=0}^N a_k e^k \right). 
\end{aligned}
$$
Observe, that, in particular, taking $s=0$ we see that the usual trace 
$\tau$ is $\hat{\beta}$ invariant.
\end{proof}

\subsubsection{$\mathrm{B2}_\theta$}

In the case of the $\Z_2$ group action, the algebra $(\T_\theta^2 \lcross_\alpha \Z_2)$ is one of 
the most studied and we have at our disposal all the necessary results.

\begin{theorem}[\cite{Kumj90}, \cite{Walt95}, \cite{ELPW}] 
$$  K_0(\T^2_\theta/\Z_2) =\Z^6, \;\;\; K_1(\T^2_\theta/\Z_2) = 0 $$
and the generators of $K_0$ group are:
$$ [1], \qq [e_{00}],\qq [e_{01}], \qq [e_{10}], \qq [e_{11}], \qq [{\cal M}_2], $$
where:
$$ 
e_{00} = \oh (1+p), \qq 
e_{01} = \oh(1+V p), \qq 
e_{10} = \oh(1+W p), \qq 
e_{11} = \oh(1+e^{i\pi\theta} VW p), 
$$
and ${\cal M}_2$ is a module, which (in the irrational case only) comes from the projection
$p_{{\cal M}_2}$ of the form $\oh e_\theta (1 + p)$ for a $\Z_2$-invariant Powers-Rieffel 
projection of trace $\theta$.
\end{theorem}
%%%%%%%%%%%%%%%%%%%%%%%%%%%%%%%%%%%%%%%%%%%%%%%%%%%%%%%%%%%%%%%%%%%%%%%%%%%%%%
To compute the explicit form of the action of the automorphism $\beta$ on the 
above generators we use the the Chern-Connes map from $K_0$. Using the usual trace
and the unbounded traces (which we wrote explicitly) one has \cite{Walt95}:
\begin{table}[here]
\centering
\begin{tabular}{|c|c|c|c|c|c||c|}
\hline
generator    & $\tau$        & $\tau_{00}$ & $\tau_{01}$ & $\tau_{10}$ & $\tau_{11}$ & C
                                                                               \\ \hline \hline
$[1]$        & $1$           & $0$         & $0$         & $0$       & $0$     &  $0$  \\ \hline
${\cal M}_2$ & $\soh \theta$ & $1$         & $-\epsilon$ & $\epsilon$&  $-1$& $1$      \\ \hline
$[e_{01}]$   & $\soh$        & $2$         & $0$         & $0$       & $0$&  $0$       \\ \hline
$[e_{10}]$   & $\soh$        & $0$         & $2$         & $0$       & $0$&  $0$      \\ \hline
$[e_{01}]$   & $\soh$        & $0$         & $0$         & $2$       & $0$&  $0$       \\ \hline
$[e_{10}]$   & $\soh$        & $0$         & $0$         & $0$       & $2$ &  $0$      \\ \hline
\end{tabular}
\caption{Value of traces on the generators of $K_0(C(\T^2_\theta) \lcross \Z_2)$}
\label{z2traces}
\end{table}
where $\epsilon$ is $+1$ for $0 < \theta < \oh$ and $-1$ for $\oh < \theta <1$. Here, $C$
denotes the canonical nontrivial cyclic two-cocycle over smooth sub algebra of 
$C(\T^2_\theta)$ (which naturally extends to its crossed product with $\Z_2$). The actual
form of the cocycle is not relevant, what matters is that its pairing with generators
of $K_0$ group is nonzero only for ${\cal M}$ (and it has been chosen to be $1$).
%%%%%%%%%%%%%%%%%%%%%%%%%%%%%%%%%%%%%%%%%%%%%%%%%%%%%%%%%%%%%%%%%%%%%%%%%%%%%%
\begin{proposition}
$$  K_0(\mathrm{B2}_\theta) = \Z^2 \oplus (\Z_2)^2, \qq \qq 
\qq K_1(\mathrm{B2}_\theta) = \Z^2.
$$
\label{lemz2}
\end{proposition}

\begin{proof}
First, combining (\ref{betatrace}) and (\ref{z2traces}) with the theorem \ref{injectK0} we 
obtain that the induced action $\hat{\beta}_*$ on the generators of $K_0$ is as follows:
$$
\begin{aligned}
&\hat{\beta}_*([e_{00}]) =  [1] - [e_{00}],&  \;\;\;
&\hat{\beta}_*([e_{10}]) =  [1] - [e_{10}],&  \\ 
&\hat{\beta}_*([e_{01}]) =  [1] - [e_{01}],&  \;\;\; 
&\hat{\beta}_*([e_{11}]) =  [1] - [e_{11}],&
\end{aligned}
$$
and
$$ 
\hat{\beta}_*([{\cal M}_2]) = [{\cal M}_2] 
           - \left( [e_{00}] - [e_{11}] \right)
           - \epsilon \left( [e_{10}] - [e_{01}] \right), \qq
\hat{\beta}_*[1] = [1]. 
$$
Using the above results we obtain the Pimsner-Voiculescu exact sequence:
$$
\xymatrix{
  0 \ar[r] 
& 0 \ar[r]
&
 K_1(\mathrm{B2}_\theta)  \ar[d] \\
  K_0(\mathrm{B2}_\theta) \ar[u]  
& \ar[l] \Z^6
& \ar[l]^{\hbox{id}-\hat{\beta}_*} \Z^6
}
$$
where $\id - \hat{\beta}_*$ has the form:
$$ \id - \hat{\beta}_* = \left( 
\begin{array}{cccccc}
0 &-1 &-1 &-1 &-1 & 0 \\
0 & 2 & 0 & 0 & 0 & 1\\
0 & 0 & 2 & 0 & 0 & \epsilon\\
0 & 0 & 0 & 2 & 0 &-\epsilon \\ 
0 & 0 & 0 & 0 & 2 &-1 \\
0 & 0 & 0 & 0 & 0 & 0 \\
\end{array} 
\right)
$$
Immediately we have:
$$ \ker (\id -\hat{\beta}_*) = \Z^2, \qq \qq \im (\id - \hat{\beta}_*) = \Z^4,$$ 
and basic algebra computations give us the result that the kernel of the
map and the quotient by its image, which are independent of the value of $\epsilon$.
\end{proof}

\subsubsection{$\mathrm{B3}_\theta$}

Here we need to use a similar type of result as the one obtained for 
the $\Z_2$ action. 

\begin{theorem}[\cite{BuWa1}] 
The $K$-theory groups and generators of $\T^2_\theta \lcross \Z_3$ are:
$$  K_0(\T^2_\theta \lcross \Z_3) =\Z^8, \;\;\; K_1(\T^2_\theta \lcross \Z_3) = 0, $$
with the generators of $K_0$ group:
$$ [1], \qq [Q_0(X)], \qq [Q_0(Y)], \qq [Q_0(p)], 
        \qq [Q_1(X)], \qq [Q_1(Y)], \qq [Q_1(p)], 
        \qq [{\cal M}_3], $$
where:
$$ Q_j(x) = \frac{1}{3}\left( 1 + e^{\frac{2\pi j i}{3}}x + e^{\frac{4\pi j i}{3}} x^2 \right),$$
and 
$$ X = e^{\frac{1}{3} \pi i \theta} Vp, \qq Y = e^{\frac{2}{3} \pi i \theta} V^2 p, $$
with $p$, being the generator of $\Z_3$ group, $p^3= \id$. The generator ${\cal M}_3$ 
corresponds to an exotic module related to the nontrivial projective module over irrational
rotation algebra.
\end{theorem}

\begin{lemma}\label{lemz3}
The action of the group $\Z_3$ on the above generators of $K$-theory is as follows:
$$
\begin{aligned}
&\hat{\beta}_*([1]) = [1], & \;\;\; 
&\hat{\beta}_*([{\cal M}_3]) = [{\cal M}_3] \!-\! [Q_0(p)] \!-\! [Q_0(X)] \!-\! [Q_0(Y)] \!+\! [1], \\
&\hat{\beta}_*([Q_1(x)]) = [Q_0(x)], & \;\;\; 
&\hat{\beta}_*([Q_0(x)]) = [1] - [Q_0(x)] -[Q_1(x)],
\end{aligned}
$$
for any $x = p, X, Y$.
\end{lemma}
\begin{proof}
Since the action is nontrivial only on the generator $p$ of the crossed
product algebra, all results concerning $Q_j(p),Q_j(X),Q_j(Y)$ are 
immediate. The only nontrivial part concerns ${\cal M}_3$.  For this we again use the 
unbounded traces and the injectivity of the associated Connes-Chern character. The
values of the traces on the generators of $K_0$ group are tabulated in 
\cite[Theorem 1.2]{BuWa1}  and from them we read out the action of $\hat{\beta}_*$.
\end{proof}

\begin{proposition}
The $K$-theory groups of $B3_\theta$ are:
$$  K_0(\mathrm{B3}_\theta) = \Z^2 \oplus \Z_3, \qq \qq 
\qq K_1(\mathrm{B3}_\theta) = \Z^2.$$
\end{proposition}

\begin{proof}
>From the Pimsner-Voiculescu exact sequence:
$$
\xymatrix{
0 \ar[r] 
& 0 \ar[r]
& K_1(\mathrm{B3}_\theta)  \ar[d] \\
  K_0(\mathrm{B3}_\theta) \ar[u]  
& \ar[l] \Z^8
& \ar[l]^{\hbox{id}-\hat{\beta}_*} \Z^8
}
$$
taking into account Lemma \ref{lemz3} we see that the matrix giving the map 
$\id - \hat{\beta}_*$ on the basis of $\Z^8$ ($[1],[Q_1(p)],[Q_0(p)],
[Q_1(X)],[Q_0(X)],[Q_1(Y)],[Q_0(Y)],[{\cal M}]$) is:

$$ \id - \hat{\beta}_* = \left( 
\begin{array}{cccccccc}
0 & 0 &-1 & 0 &-1 & 0 &-1 & -1 \\
0 & 1 & 1 & 0 & 0 & 0 & 0 & 0 \\
0 &-1 & 2 & 0 & 0 & 0 & 0 &1 \\
0 & 0 & 0 & 1 & 1 & 0 & 0 & 0 \\ 
0 & 0 & 0 &-1 & 2 & 0 & 0 &1 \\
0 & 0 & 0 & 0 & 0 & 1 & 1 & 0 \\
0 & 0 & 0 & 0 & 0 &-1 & 2 &1 \\
0 & 0 & 0 & 0 & 0 & 0 & 0 & 0
\end{array} 
\right)
$$
\end{proof}
%%%%%%%%%%% HERE %%%%%%%%%%%%%%%%%%%%%%%%%%%%%
\subsubsection{$\mathrm{B4}_\theta$}

Let us begin with the following result:

\begin{lemma}[Theorem 2.1 \cite{Walt01},\cite{Walt00}]
The $K$-groups of $\T^2_\theta \rtimes \Z_4$ are
$$  K_0(\T^2_\theta \rtimes \Z_4) =\Z^9, \;\;\; K_1(\T^2_\theta \rtimes \Z_4) = 0. $$
The generators are:
$$
\begin{aligned} 
\qq  [1], \qq [Q_0(p)], \qq [Q_1(p)], \qq [Q_2(p)], 
       \qq  [Q_0(e^{\frac{\pi i \theta}{2}} Vp)], \\
[Q_1(e^{\frac{\pi i \theta}{2}} Vp)], \qq [Q_2(e^{\frac{\pi i \theta}{2}} Vp)], \qq 
[Q_0(Vp^2)], \qq [{\cal M}_4], 
\end{aligned}
$$
where
$$ Q_k(x) =\frac{1}{4}(1+ (i^k x) + (i^k x)^2 + (i^k x)^3), \qq k=0,1,2, $$
and ${\cal M}_4$ is the nontrivial module arising from the nontrivial projective 
module over the noncommutative torus.
\end{lemma}
      
Again, using an injective morphism coming from the Chern-Connes character from 
$K_0(\T^2_\theta \lcross \Z_4)$ to $\R^5 \times \C^2$ (that does not come as a surprise 
as the action of $\hat{\beta}$ is in case of some of the unbounded traces multiplication 
by $\pm i$) an the explicit computation of the traces \cite[page 640]{Walt00} we obtain 
the following result:

\begin{lemma}\label{lemz4}
The action of the group $\Z$ on the above generators of $K$-theory is:
\begin{align*}
&\hat{\beta}_*([1]) = [1], \qq \qq
& \hat{\beta}_*([Q_0(Vp^2)]) = [1] - [Q_0(Vp^2)], \\
&\hat{\beta}_*([Q_2(x)]) = [Q_1(x)], 
&\hat{\beta}_*([Q_0(x)]) =[1] - [Q_0(x)] - [Q_1(x)] - [Q_2(x)],  \\
&\hat{\beta}_*([Q_1(x)]) = [Q_0(x)], 
&\hat{\beta}_*([{\cal M}_4]) = [{\cal M}_4] \!-\! [Q_0(Vp^2)] \!-\! [Q_0(p)] \!-\! 
[Q_0(e^{\frac{\pi i \theta}{2}}Vp)]\!+\![1], 
\end{align*}

for any $x = p, e^{\frac{\pi i \theta}{2}} Vp$.
\end{lemma}

\begin{proposition}
The $K$-theory groups of $B4_\theta$ are:
$$  K_0(\mathrm{B4}_\theta) = \Z^2 \oplus \Z_2, \qq \qq 
\qq K_1(\mathrm{B4}_\theta) = \Z^2,
$$
\end{proposition}

\begin{proof}
>From the Pimsner-Voiculescu exact sequence:
$$
\xymatrix{
0 \ar[r] 
& 0 \ar[r]
& K_1(\mathrm{B4}_\theta)  \ar[d] \\
  K_0(\mathrm{B4}_\theta) \ar[u]  
& \ar[l] \Z^9
& \ar[l]^{\hbox{id}-\hat{\beta}_*} \Z^9
}
$$
using (\ref{lemz4}) we see that the matrix giving $\id - \hat{\beta}_*$ on 
the basis of $K_0(C(T^2_\theta) \rtimes \Z_4)$ 
($[1],[Q_2(p)],[Q_1(p)],[Q_0(p)],[Q_2(e^{\frac{\pi i \theta}{2}}Vp)],
[Q_1(e^{\frac{\pi i \theta}{2}}Vp)],[Q_0(e^{\frac{\pi i \theta}{2}}Vp)] ,[Q_0(Vp^2)],[{\cal M}]$)
is:

$$ \id - \hat{\beta}_* = \left( 
\begin{array}{ccccccccc}
0 & 0 & 0 & -1 & 0 & 0 & -1 &-1 &-1\\
0 & 1 & 0 & 1  & 0 & 0 & 0 & 0 &  0\\
0 &-1 & 1 & 1 & 0 & 0 & 0 & 0 & 0\\
0 & 0 &-1 & 2 & 0 & 0 & 0 & 0 & 1\\ 
0 & 0 & 0 & 0 & 1 & 0 & 1 & 0 & 0\\
0 & 0 & 0 & 0 &-1 & 1 & 1 & 0 & 0\\
0 & 0 & 0 & 0 & 0 &-1 & 2 & 0 & 1\\
0 & 0 & 0 & 0 & 0 & 0 & 0 & 2 & 1\\
0 & 0 & 0 & 0 & 0 & 0 & 0 & 0 & 0
\end{array} 
\right)
$$
\end{proof}
%%%%%%%%%%%%%%%%%%%%%%%%%%%%%%%%%%%%%%%%%%%%%%%%%%%%%%%%%%%%%%%%%%%%%%%%
%%%%%%%%%%%%%%%%%%%%%%%%%%%%%%%%%%%%%%%%%%%%%%%%%%%%%%%%%%%%%%%%%%%%%%%%
\subsubsection{$\mathrm{B6}_\theta$}

Similarly as in the case of cubic transform we use the results of hexic
transform \cite{BuWa1,BuWa2}.

\begin{theorem}[\cite{BuWa1}, Theorem 1.1] 
The $K$-theory groups and generators of $C(T^2_\theta) \lcross \Z_6$ are:
$$  K_0(\T^2_\theta \lcross \Z_6) =\Z^{10}, \;\;\; K_1(\T^2_\theta \lcross \Z_6) = 0, $$
with the generators of $K_0$ group:
$$ [1], [{\cal M}_6], [Q_0(e^{\frac{\pi i}{3}} V p^2)], [Q_2(e^{\frac{\pi i}{3}} V p^2)], [Q_0(Vp^3)], 
[Q_n(p)], n=0,1,2,3,4,$$ 
where:
$$ Q_n(x) = \frac{1}{6} \sum_{k=0}^5  e^{\frac{2\pi n k i}{3}} x^k, \qq\qq n=0,1,2,3,4. $$
and the generator ${\cal M}_6$ is again the exotic one.
\end{theorem}

\begin{lemma}\label{lemz6}
The action of the group $\Z$ through $\hat{\beta}_*$ on the above generators 
of $K$-theory is:
$$
\begin{aligned}
& \hat\beta_*([1]) = [1], &&\hat\beta_*([Q_0(Vp^3)]) = [1] - [Q_0(Vp^3)], \\
&\hat\beta_*([Q_{n+1}(p)]) = [Q_n(p)], \qq n=1,2,3,4, \;\;\;\; &&\hat\beta_*([Q_2(e^{\frac{\pi i}{3}} V p^2) = [Q_0(e^{\frac{\pi i}{3}} V p^2)], \\
&\hat\beta_*([Q_0(p)) = [1] - \sum_{k=0}^4 [Q_k(p)], &&
\\
\end{aligned}
$$
$$
\begin{aligned}
&\beta_*([{\cal M}_6]) = [{\cal M}_6]  - [Q_0(p)] -[Q_0(e^{\frac{\pi i}{3}} V p^2)] - [Q_0(Vp^3)] +[1], \\
&\beta_*([Q_0(e^{\frac{\pi i}{3}} V p^2)]) = [1]-[Q_0(e^{\frac{\pi i}{3}} V p^2)] 
- [Q_2(e^{\frac{\pi i}{3}} V p^2)], 
\end{aligned}
$$
\end{lemma}
\begin{proof}
Again the action is immediate to read on the generators $Q_n(p)$, whereas using the 
property of the twisted traces, their behavior under $\hat{\beta}$ and the explicit
table giving giving the values of these traces on the generators \cite[Theorem 1.1]{BuWa1}
we obtain the relations above, in particular the highly nontrivial part concerns $[{\cal M}_6]$. 
\end{proof}
%%%%%%%%%% HERE HERE %%%%%%%%%%%%%%%%%%%%%%%%%%%%%%
\begin{proposition}
The $K$-theory groups of ${\mathrm B6}_\theta$ are:
$$  K_0(\mathrm{B6}_\theta) = \Z^2 , \qq \qq 
\qq K_1(\mathrm{B6}_\theta) = \Z^2. $$
\end{proposition}

\begin{proof}
>From the Pimsner-Voiculescu exact sequence:
$$
\xymatrix{
0 \ar[r] 
& 0 \ar[r]
& K_1(\mathrm{B5}_\theta)  \ar[d] \\
  K_0(\mathrm{B5}_\theta) \ar[u]  
& \ar[l] \Z^{10}
& \ar[l]^{\hbox{id}-\beta_*} \Z^{10}
}
$$
taking into account (\ref{lemz6}) we see that the matrix giving the map $\id - \beta_*$
on the basis of $K_0(C(\T^2_\theta) \lcross \Z_6$ (in the following order) 
$([1], [Q_4(p)], [Q_3(p)],[Q_2(p)],[Q_1(p)],$ $[Q_0(p)],[Q_2(e^{\frac{\pi i}{3}} V p^2)],
[Q_0(e^{\frac{\pi i}{3}} V p^2)], [Q_0(Vp^3)],[{\cal M}_6])$
is:

$$ \id - \beta_* = \left( 
\begin{array}{cccccccccc}
%1   
0 & 0 & 0 & 0 & 0 &-1 & 0 &-1 &-1 & -1\\
%p0  
0 & 1 & 0 & 0 & 0 & 1 & 0 & 0 & 0 & 0\\
%p1  
0 &-1 & 1 & 0 & 0 & 1 & 0 & 0 & 0 &0\\
%p2  
0 & 0 &-1 & 1 & 0 & 1 & 0 & 0 & 0 & 0\\
%p3  
0 & 0 & 0 &-1 & 1 & 1 & 0 & 0 & 0 & 0\\ 
%p4  
0 & 0 & 0 & 0 &-1 & 2 & 0 & 0 & 0 & 1\\
%q0  
0 & 0 & 0 & 0 & 0 & 0 & 1 & 1 & 0 & 0\\
%q1  
0 & 0 & 0 & 0 & 0 & 0 &-1 & 2 & 0 &1\\
%r   
0 & 0 & 0 & 0 & 0 & 0 & 0 & 0 & 2 & 1\\
%M6  
0 & 0 & 0 & 0 & 0 & 0 & 0 & 0 & 0 & 0
\end{array} 
\right)
$$
\end{proof}

\section{K-theory of classical Bieberbach manifolds}

Although the computations we have presented were for the specific case of an irrational
value of $\theta$ the method works , slightly modified, for the rational case.. In particular, 
the $K$-theory groups of $C(\T^2_\theta) \lcross
 \Z_n$ and their generators are independent 
of $\theta$ remain unchanged (see \cite{ELPW}), which follows from the fact that these are 
twisted group algebras and their K-theory groups depend on the homotopy class of twisting 
cocycle, which in this case are, of course, trivial. 

Clearly, the explicit form of the generator of the nontrivial module over $C(\T^2_\theta)$ very 
much depends on whether $\theta$ is rational. The crucial difference between the rational and 
irrational case is that to have an injective morphism from the $K_0$ group of $C(\T^2_\theta) 
\lcross \Z_N$ into $\R^{r(N)}$ one needs to use the nontrivial Chern-Connes character (called 
second-order Chern character by Walters in \cite{Walt00})
coming from the cyclic two-cocycle over $\T^2_\theta$.

For $\Z_2$ the original result of Walters \cite{Walt95} is for the value of $\theta \in \R \setminus \Q$ 
but it is easy to see that the arguments are valid as well for rational $\theta$. 
For $\Z_4$ the result in \cite{Walt00} is valid for all $\theta \in \R$, and for $\Z_3$ 
the results are combined in the papers of Buck and Walters \cite{BuWa1} 
and \cite{BuWa2}, where, again, they are obtained for any value of $\theta$, rational 
or irrational. 

Since the nontrivial Chern-Connes character vanishes on all generators of $K_0$ group 
apart from the nontrivial one (which is called a Bott class in \cite{ELPW} and Fourier 
module by Walters in \cite{Walt00}). For our purpose, the crucial information is the behavior of this character
under the action of $\hat{\beta}$. We have:

\begin{lemma}
For any $N=2,3,4,6$, the Chern-Connes character induced by the cyclic 2-cocycle over 
$C(\T^2_\theta)$ is invariant under the action of $\hat{\beta}$. 
\end{lemma}

The proof is trivial: since the action of $\beta$ does not change $C(\T^2_\theta)$ and the trace
on it as well as derivations are invariant, so must be the nontrivial Chern-Connes character. 
Therefore, $\hat{\beta}_*$ of the nontrivial generator of $K_0$ group (which we called ${\cal M}_N$
for $N=2,3,4,6$ must be a sum of ${\cal M}_N$ with a combination of remaining generators, as
 is clearly the case in proof of proposition \ref{lemz2} and lemmas \ref{lemz3}, \ref{lemz4}, \ref{lemz6}.
 
\section{Conclusions}

The Bieberbach manifolds as well as their noncommutative versions are quite interesting
objects and a good testing ground for tools of noncommutative geometry \cite{OlSi}. 
The computations presented above are first results on their $K$-theory groups and use 
techniques developed to study noncomutative tori and their crossed product by the action 
of a cyclic group. Although a noncommutative torus is a Rieffel-type deformation of the
classical object, the noncommutative Bieberbachs are not. Therefore, the result that their
$K$-theory groups are the same in the deformed and undeformed case is not trivial.

The result itself is not entirely surprising: the $K_0$ groups have torsion, although the exact
form of the torsion component as well as the fact that n the $\Z_6$ case there is no torsion
cannot be seen at once. It is a remarkable fact that there exists a striking relation 
between the $K_0$ groups of the manifolds $\mathrm{BN}$ ($N=2,3,4,6$) and the first homology 
groups of the corresponding infinite Bieberbach groups $G_N$ \cite{HJKL93}, (so that 
$\mathrm{BN} =\R^3/G_N$), namely $K_0(\mathrm{BN}) \sim \Z \oplus H_1(G_N,\Z)$.  This fact, as well the remaining case of nonorientable manifolds, together with the study of spectral geometries and spin structures over noncommutative Bieberbach manifolds shall be discussed in our future work.
%%%%%%%%%%%%%%%%%%%%%%%%%%%%%%%%%%%%%%%%%%%%%%%%%%%%%%%%%%%%%%%%%%%
%%%%%%%{\bf Acknowledgements:}
%%%%%%%%%%%%%%%%%%%%%%%%%%%%%%%%%%%%%%%%%%%%%%%%%%%%%%%%%%%%%%%%%%%
 
\end{document}